\documentclass[12pt]{amsart}

\usepackage{amsfonts,amssymb,stmaryrd,amscd,amsmath,latexsym,amsbsy}

\newtheorem{theorem}{Theorem}[section]
\newtheorem{lemma}[theorem]{Lemma}
\newtheorem{proposition}[theorem]{Proposition}
\newtheorem{corollary}[theorem]{Corollary}
\theoremstyle{definition}
\newtheorem{definition}[theorem]{Definition}

\newtheorem{remark}[theorem]{Remark}

\newcommand{\Tr}{\text{Tr}}%{\text{Tr}\,}
\newcommand{\id}{\text{id}}

\newcommand{\Hom}{\text{Hom}}

\newcommand{\Rep}{\text{Rep}}

\newcommand{\g}{\mathfrak{g}}

\newcommand{\Dim}{{\rm Dim}}
\newcommand{\kk}{{\bold k}}

\newcommand{\C}{\mathcal{C}}

\newcommand{\ben}{\begin{enumerate}}
\newcommand{\een}{\end{enumerate}}

\newcommand{\be}{{\bf 1}}

\theoremstyle{plain}

\newtheorem*{sol}{Solution}

\theoremstyle{definition}

\theoremstyle{remark}

\newcommand{\solu}[1]{\begin{sol}{\bf (\ref{#1})}}

\def\g{\mathfrak{g}}

\def\C{\mathcal{C}}

\def\Q{\mathbb{Q}}

\def\Hom{\mathrm{Hom}}

\def\Dim{{\rm Dim}}

\def\B{\mathcal{B}}

\def\A{\mathcal{A}}

\def\Vec{\mathrm{Vec}}
\def\sVec{\mathrm{sVec}}
\def\Ver{\mathrm{Ver}}

\def\Rep{\mathop{\mathrm{Rep}}\nolimits}

\def\SFPdim{\mathop{\mathrm{SFPdim}}\nolimits}
\def\FPdim{\mathop{\mathrm{FPdim}}\nolimits}
\def\sdim{\mathop{\mathrm{sdim}}\nolimits}

\begin{document}

\title{Computations in symmetric fusion categories in characteristic $p$}

\author{Pavel Etingof, Victor Ostrik, Siddharth Venkatesh}

\address{Etingof: Department of Mathematics, Massachusetts Institute of Technology, Cambridge, MA 02139,
USA}

\email{etingof@math.mit.edu}

\address{Ostrik: Department of Mathematics
University of Oregon
Eugene, OR 97403, USA}

\email{vostrik@math.uoregon.edu}

\address{Venkatesh: Department of Mathematics, Massachusetts Institute of Technology, Cambridge, MA 02139, USA}

\email{sidnv@mit.edu}

\begin{abstract} 
We study properties of symmetric fusion categories in characteristic $p$. In particular, we introduce the notion of a super Frobenius-Perron dimension of an object $X$ of such a category, and derive an explicit formula for the Verlinde fiber functor 
$F(X)$ of $X$ (defined by the second author) in terms of the usual and super Frobenius-Perron dimensions of $X$. 
We also compute the decomposition of symmetric powers of objects of the Verlinde category, generalizing a classical formula of Cayley and Sylvester for invariants of binary forms. Finally, we show that the Verlinde fiber functor is unique, and classify braided fusion categories of rank two and triangular semisimple Hopf algebras in any characteristic. 
\end{abstract}

\maketitle

\section{Introduction}
Let $\kk$ be an algebraically closed field of characteristic $p>0$. 
 Let $\C$ be a symmetric fusion category over $\kk$. 
Then, according to the main result of \cite{O}, $\C$ admits a symmetric 
tensor functor $F: \C\to {\rm Ver}_p$ into the Verlinde category ${\rm Ver}_p$
(the quotient of ${\rm Rep}_\kk(\Bbb Z/p\Bbb Z)$ by negligible morphisms); we prove that this functor is unique. 
We derive an explicit formula for the decomposition of $F(X)$ into simple objects for each $X\in \C$; 
it turns out that this decomposition is completely determined by just two parameters
--- the Frobenius-Perron dimension $\FPdim(X)$ and the super Frobenius-Perron dimension $\SFPdim(X)$
which we introduce in this paper. We use this formula to find the decomposition of symmetric powers 
of objects in ${\rm Ver}_p$, and in particular find the invariants in them --- the ``fusion'' analog of the classical formula for polynomial invariants of binary forms. We also relate the super Frobenius-Perron dimension 
to the second Adams operation and to the $p$-adic dimension introduced in \cite{EHO}, and classify symmetric categorifications of fusion rings of rank 2. Finally, we classify triangular semisimple Hopf algebras in an arbitrary characteristic, generalizing the result of \cite{EG} for characteristic zero, and classify braided fusion categories of rank two. 

The paper is organized as follows. Section 2 contains preliminaries. In Section 3 we prove the uniqueness of the Verlinde fiber functor. In Section 4, we define the super Frobenius-Perron dimension of an object of a symmetric fusion category, and 
give a formula for this dimension in terms of the second Adams operation; here we also classify braided fusion categories of rank two. In Section 5, we prove a decomposition formula 
for the Verlinde fiber functor of an object $X$ of a symmetric fusion category in terms of the ordinary and super Frobenius-Peron dimensions of $X$, and use this formula to compute the transcendence degrees of the symmetric and exterior algebra of $X$. In Section 6, we use the formula of Section 5 to find the decomposition of symmetric powers 
of objects in the Verlinde category, and in particular give a formula for their invariants. In Section 7, we compute the $p$-adic dimensions of an object in a symmetric fusion category. Finally, in Section 8, we classify semisimple triangular Hopf algebras 
in any characteristic. 

{\bf Acknowledgements.} The work of P.E. was partially supported by the NSF grant DMS-1502244.

\section{Preliminaries}

Throughout the paper, $\kk$ will denote an algebraically closed field. 
Unless specified otherwise, we will assume that the characteristic of $\kk$ is
$p>0$.  

By a symmetric tensor category over $\kk$ (of any characteristic) we will mean an artinian rigid symmetric monoidal category
in which the tensor product is compatible with the additive structure, and ${\rm End}(\bold 1)=\kk$ (see \cite{EGNO}, Definitions 4.1.1, 8.1.2). By a symmetric fusion category we mean a semisimple symmetric tensor category 
with finitely many simple objects. 
 
\subsection{Verlinde categories}  
Recall that the {\it Verlinde category} ${\rm Ver}_{p}$ is a symmetric fusion category over $\kk$
obtained as the quotient of $\Rep_{\kk}({\Bbb Z}/p{\Bbb Z})$ by the tensor ideal of 
negligible morphisms, i.e. morphisms $f: X\to Y$ such that for any $g: Y\to X$ one has ${\rm Tr}(fg)=0$
(see \cite{O} for details). This category has $p-1$ simple objects, $\mathbf{1} = L_{1}, \ldots, L_{p-1}$, such that
$$
L_{r} \otimes L_{s} \cong \sum_{i=1}^{\min(r, s, p-r, p-s)} L_{|r -s| + 2i - 1}
$$
(the Verlinde fusion rules). 

\begin{definition} 
We define ${\rm Ver}_{p}^{+}$ to be the abelian subcategory of ${\rm Ver}_p$ generated by $L_{i}$ for $i$ odd, 
and define ${\rm Ver}_{p}^{-}$ to be the abelian subcategory of ${\rm Ver}_p$ generated by $L_{i}$ for $i$ even.
\end{definition}

By the Verlinde fusion rules, ${\rm Ver}_{p}^{+}$ is a fusion subcategory of ${\rm Ver}_{p}$, 
and tensoring with $L_{p-1}$ gives an equivalence of abelian categories ${\rm Ver}_{p}^{+} \rightarrow {\rm Ver}_{p}^{-}$ as long as $p > 2$. Since for $p>2$ the symmetric fusion subcategory generated by $L_{1}$ and $L_{p-1}$ is ${\rm sVec}$, the category of supervector spaces over $\kk$, we see that 
$$
{\rm Ver}_{p} = {\rm Ver}_{p}^{+} \oplus {\rm Ver}_{p}^{-} \cong {\rm Ver}_{p}^{+} \boxtimes {\rm sVec},\ p>2
$$
(see \cite[3.3]{O}).

Recall that for $r \in \mathbb{Z}$ and $z \in \mathbb{C}$, 
$$
[r]_{z} = \frac{z^{r} - z^{-r}}{z - z^{-1}}.
$$
Let $q = e^{\pi i/p}$ be a primitive $2p$-th root of unity, and let $\mathrm{Gr}({\rm Ver}_{p})$ denote the Grothendieck ring of ${\rm Ver}_{p}$. 
The following lemma is standard. 

\begin{lemma} \label{char:Ver}
$\mathrm{Gr}({\rm Ver}_{p})$ has exactly $p-1$ distinct characters $\chi_{1}, \ldots, \chi_{p-1}$, where 
$$
\chi_{j}(L_{r}) = [r]_{q^{j}}.
$$
\end{lemma}

\begin{proof} It is well known (and easy to show) that $\chi_{j}$ is a character. That these are all the distinct characters 
follows from the facts that $\mathrm{Gr}({\rm Ver}_{p}) \otimes \mathbb{C}$ has dimension $p-1$ and that distinct characters 
are linearly independent.
\end{proof}

Note that the character $\chi_1$ takes only positive values, and therefore coincides with 
the Frobenius-Perron dimension ${\rm FPdim}$. 

Let $K=\Q(q)\cap \Bbb R$. Then $K$ is a Galois extension of $\Q$, with
$$
{\rm Gal}(K/\Q) = \{g_{0}, \ldots, g_{\frac{p-3}{2}}\} \cong {\Bbb Z}/\frac{p-1}{2}{\Bbb Z}
$$
for $p>2$, where $g_{s}$ is defined on $\Bbb Q(q)$ by sending $q$ to $q^{2s+1}$; 
for $p=2$, we have $K=\Bbb Q$. 

\begin{proposition} \label{Galois} The characters $\chi_j$ of $\mathrm{Gr}({\rm Ver}_{p})$ land in $K$. Moreover, for $p>2$ and $0\le s\le \frac{p-3}{2}$ we have 
$$
\chi_{2s+1} = g_{s} \circ \chi_1, \;\; \chi_{p-2s-1} = g_{s}\circ \chi_{p-1}.
$$ 
\end{proposition}

\begin{proof} 
It is clear from Lemma \ref{char:Ver} that all values of $\chi_j$ lie in $\Bbb Q(q)$ and are real, so the first statement follows. 
The second statement is obvious. 
\end{proof} 

Recall that for any object $X$ of a symmetric tensor category, we can define its symmetric powers 
$S^iX$ and exterior powers $\wedge^iX$ (\cite{EHO}, 2.1). 

\begin{proposition}\label{prop31} For every $r>1$, we have $S^iL_r=0$ if $i+r>p$, 
and for every $r<p-1$, we have $\wedge^iL_r=0$ if $i>r$. 
\end{proposition} 

\begin{proof} The first statement is proved in the proof of Proposition 3.1 of \cite{Ven}. 
The second statement follows from the first one using the equality
$$
\wedge^iL_r=L_{p-1}^{\otimes i}\otimes S^i(L_{p-1}\otimes L_r)=L_{p-1}^{\otimes i}\otimes S^iL_{p-r},
$$
which holds since $L_{p-1}$ generates a copy of ${\rm sVec}$ inside ${\rm Ver}_p$. 
\end{proof} 

\subsection{The Verlinde fiber functor}\label{vff} 

Let $\C$ be any symmetric fusion category over $\kk$. 
The following theorem is the main result of \cite{O}:

\begin{theorem}\label{fiber} (\cite{O}) There exists a symmetric tensor functor \linebreak $F: \C\to {\rm Ver}_p$.  
\end{theorem}   

We also have the following result, proved in Section \ref{uniqueness}.

\begin{theorem} \label{ufiber}
 A symmetric tensor functor $F: \C\to {\rm Ver}_p$ is unique up to a non-unique isomorphism of tensor functors.
\end{theorem}   

This theorem generalizes Theorem 3.2 of \cite{DM} for functors with 
values in ${\rm Vec}$ and its extension to ${\rm sVec}$ discussed in Section 3 of \cite{De}, and its proof given below is based on similar ideas. 
 
The functor $F$ is called the {\it Verlinde fiber functor}, as it generalizes the fiber functor 
$\C\to \sVec$ in characteristic zero, whose existence is a theorem of Deligne, \cite{De} (see also \cite{EGNO}, Theorem 9.9.26). 

The functor $F$ is constructed as follows (see \cite{O} for more details). First one constructs 
the {\it Frobenius functor} -- a symmetric tensor functor ${\rm Fr}: \C\to \C^{(1)}\boxtimes {\rm Ver}_p$, where $\C^{(1)}$ denotes the Frobenius twist of $\C$. 
Iterating this functor and taking tensor product in ${\rm Ver}_p$, one obtains the $n$-th Frobenius functor 
${\rm Fr}^n: \C\to \C^{(n)}\boxtimes {\rm Ver}_p$, where $\C^{(n)}$ is the $n$-th Frobenius twist of $\C$
(where ${\rm Fr}^1={\rm Fr}$). One then shows that for a sufficiently large $n$, ${\rm Fr}^n$ lands 
in ${\mathcal D}\boxtimes {\rm Ver}_p$, where ${\mathcal D}\subset \C^{(n)}$ is a nondegenerate fusion category. 
Then, using lifting theorems to characteristic zero (see \cite{EGNO}, Subsection 9.16) and Deligne's theorem in characteristic zero \cite{De}, 
one shows that there exists a symmetric tensor functor $\Phi: {\mathcal D}\to {\rm sVec}$. 
Then 
$$
(\Phi\boxtimes {\rm Id})\circ {\rm Fr}^n: \C\to {\rm sVec}\boxtimes {\rm Ver}_p, 
$$
and we get $F$ by composing this functor with the projection 
$$
\otimes: {\rm sVec}\boxtimes {\rm Ver}_p\to {\rm Ver}_p
$$ 

\begin{remark} In fact, it is shown in \cite{O} that the functor ${\rm Fr}^n$ lands in $\C^{(n)}\boxtimes {\rm Ver}_p^+$, 
so that $\Phi$ lands in ${\rm sVec}\boxtimes {\rm Ver}_p^+$.  
\end{remark} 

\begin{corollary} If $\C$ is a symmetric fusion category over $\kk$, then for any $X \in \C$, 
the Frobenius-Perron dimension $\FPdim(X)$ belongs to ${\Bbb Z}[q]\cap K.$ 
\end{corollary}

\begin{proof} By Theorem \ref{fiber}, we have the Verlinde fiber functor $F: \C\to {\rm Ver}_p$. 
This functor preserves $\FPdim.$ Thus, for $X \in \C$
$$
\FPdim(X) = \FPdim(F(X)) = \chi_{1}(F(X)) \in \Bbb Z[q].
$$
Also, $\FPdim(X)\in \Bbb R$, which implies the statement. 
\end{proof}

\subsection{$p$-adic dimensions} 

Let $\C$ be a symmetric tensor category over $\kk$. Let $X\in \C$. Then to $X$ one can attach its categorical dimension $\dim(X)\in \kk$, and it is shown in \cite{EGNO}, Exercise 9.9.9(ii) that in fact $\dim(X)\in \Bbb F_p\subset \kk$ (see also \cite{EHO}, Lemma 2.2). 

Further, it is shown in \cite{EHO} that
$\dim(X)$ is the reduction modulo $p$ of a richer invariant of $X$ which takes values in $\Bbb Z_p$, called the  $p$-adic dimension of $X$. In fact, there are two such invariants: the symmetric $p$-adic dimension 
$\Dim_+(X)$ and the exterior $p$-adic dimension $\Dim_-(X)$.  

Namely, let $d=d_0+pd_1+p^2d_2+...$ be a $p$-adic integer (where $d_i$ are the digits of $d$).
We can define the power series over $\Bbb F_p$,
$$
(1+z)^d:=(1+z)^{d_0}(1+z^p)^{d_1}(1+z^{p^2})^{d_2}...
$$

\begin{definition}(\cite{EHO}) We say that the symmetric $p$-adic dimension $\Dim_+(X)$ of $X$ is $d$ if
$$
\sum_j \dim(S^jX)z^j=(1-z)^{-d},
$$
and say that the exterior $p$-adic dimension $\Dim_-(X)$ of $X$ is $d$ if 
$$
\sum_j \dim(\wedge^jX)z^j=(1+z)^d.
$$
\end{definition} 

\begin{theorem} (\cite{EHO}) The  p-adic dimensions $\Dim_+(X)$, $\Dim_-(X)$ exist and are unique for any $X$. 
\end{theorem} 

\section{Uniqueness of the Verlinde fiber functor}\label{uniqueness} 
The goal of this Section is to give a proof of Theorem \ref{ufiber}.

\subsection{Commutative algebras in ${\rm Ver}_{p}$}
\begin{lemma}\label{snilp}
 Let $X$ be an object of ${\rm Ver}_{p}$ with $\Hom(\be,X)=0$. Then there exists $N$ such that
$S^nX=0$ for all $n\ge N$.
\end{lemma}

\begin{proof} We proceed by induction in the length of $X$. Then the base of induction is given by Proposition \ref{prop31}
and the induction step follows from the decomposition
$$S^n(X\oplus Y)=\bigoplus_{i=0}^nS^iX\otimes S^{n-i}Y.$$
\end{proof} 

Now let $A$ be a unital commutative associative algebra in the category ${\rm Ver}_{p}$, see e.g. \cite[Section 8.8]{EGNO}.
Let $m_n: A^{\otimes n}\to A$ be the multiplication morphism (so $m_1=\id_A$, $m_2=m$ is the multiplication in $A$, and
$m_n=m\circ (m_{n-1}\otimes \id_A)$ for $n\ge 2$). There is a unique decomposition $A=A_0\oplus A_1$ where $A_0$
is a multiple of $\be$ and $\Hom(\be,A_1)=0$.

\begin{lemma} \label{anilp}
 There exists $N$ such that the restriction of $m_n$ to $A_1^{\otimes n}$ equals zero for $n\ge N$.
\end{lemma}

\begin{proof} Since $A$ is commutative, the restriction of $m_n$ to  $A_1^{\otimes n}$ factors through the canonical map
 $A_1^{\otimes n}\to S^nA_1$. The result follows from Lemma \ref{snilp}.
\end{proof}

Let $J$ be the image of multiplication morphism $m: A\otimes A_1\to A$. Then $J$ is an ideal in $A$
(note that since $A$ is commutative, there is no difference between left, right and two-sided ideals 
in $A$). The ideal $J$ is nilpotent:

\begin{lemma} \label{inilp}
 There exists $N$ such that the restriction of $m_n$ to $J^{\otimes n}$ equals zero for $n\ge N$.
\end{lemma}

\begin{proof} Since $A$ is associative, the image of $m_n$ restricted to $J^{\otimes n}$ 
equals the image of $m_{2n}$ restricted to $(A\otimes A_1)\otimes \ldots \otimes (A\otimes A_1)$. Since $A$ is commutative, this is the same as image of $m_{2n}=m\circ (m_n\otimes m_n)$
restricted to $A^{\otimes n}\otimes A_1^{\otimes n}$. The result follows from Lemma \ref{anilp}.
\end{proof}

The unit object $\be \in {\rm Ver}_{p}$ has an obvious structure of unital, asocialtive, commutative
algebra.

\begin{proposition} \label{homto1}
Let $A\in {\rm Ver}_{p}$ be a unital, associative, commutative algebra.
Then there exists a unital algebra homomorphism $A\to \be$.
\end{proposition}

\begin{proof} Since $A$ is unital, the ideal $J$ contains $A_1$. By Lemma \ref{inilp} this ideal
is proper: $J\ne A$. Thus $A/J$ is a nonzero unital, associative, commutative algebra contained
in subcategory $\Vec \subset {\rm Ver}_{p}$ generated by the unit object. Thus by Hilbert's
Nullstellensatz there exists a unital homomorphism $A/J\to \be$. The result follows.
\end{proof}

\subsection{Fiber functors and commutative algebras} \label{commalg}
Let $\C$ be a symmetric fusion category and let $F: \C \to {\rm Ver}_{p}$ be a symmetric tensor
functor. Let $I: {\rm Ver}_{p}\to \C$ be the right adjoint of $F$ (it exists since the categories $\C$
and ${\rm Ver}_{p}$ are finite semisimple). 

\begin{proposition}(cf. \cite[Lemma 3.5]{DMNO}) \label{falg}

(i) The object $A:=I(\be)\in \C$ has a natural structure of unital,
associative, commutative algebra in $\C$. 

(ii) The functor $I$ induces a tensor equivalence $\tilde I: F(\C)\simeq \C_A$ 
of the image $F(\C)$ of $F$ in ${\rm Ver}_{p}$ (i.e., the fusion subcategory generated by objects of the form $F(X)$) 
with the category $\C_A$ of right $A-$modules in $\C$.

(iii) The functor $\tilde I \circ F: \C \to \C_A$ is naturally isomorphic to the functor
$F_A=?\otimes A$ as a tensor functor.
\end{proposition}

\begin{proof} Part (i) follows from Lemma 3.5 of \cite{DMNO}, see \cite[Remark 3.7]{DMNO}.
The functor $F$ makes ${\rm Ver}_{p}$ into a module category over $\C$, see e.g.
\cite[Example 7.4.7]{EGNO}. It follows from definitions that the functor $I$ identifies
with the internal Hom functor $\underline{\Hom}(\be,?)$. Thus
the fact that $I$ induces an equivalence $\tilde I$ follows from \cite[Corollary 7.10.5 (ii)]{EGNO}.
Also using \cite[(7.25)]{EGNO} we get a canonical isomorphism of functors in part (iii). Now
one defines tensor structure on the functor $\tilde I$ as in the proof of Theorem 3.1 of
\cite{ENO}; it follows immediately from definitions that the isomorphism of functors in (iii)
respects this tensor structure.
%Parts (i),(ii) follow from the proof of Lemma 3.5 of 
%\cite{DMNO}. Part (iii) follows from the fact that for any $X\in \C$
%we have a canonical isomorphism $I(F(X))\cong X\otimes I(\bold 1)=X\otimes A$, and this isomorphism is compatible with tensor products. 
\end{proof} 

Note that Proposition \ref{falg} (ii) implies that $A=\tilde I(\be)$ is a simple $A-$module. 
Equivalently we have 

\begin{corollary} \label{noid}
The algebra $A$ has no nontrivial ideals. 
\end{corollary}

Let $\tilde F: \C_A\to {\rm Ver}_{p}$ denote the composition of $\tilde I^{-1}$ with the
embedding functor $F(\C)\to {\rm Ver}_{p}$. Note that functor $\tilde F$ is injective
(that is, fully faithful). Then Proposition \ref{falg} (iii) implies

\begin{corollary} \label{compo}
The functor $F$ is isomorphic to the composition $\tilde F\circ F_A$ as a tensor functor.
\end{corollary}

\subsection{Autoequivalences of ${\rm Ver}_{p}$}
The following result contains a special case of Theorem \ref{ufiber} when $\C={\rm Ver}_{p}$.

\begin{lemma} \label{autover}
Any tensor autoequivalence of ${\rm Ver}_{p}$ (not necessarily preserving the braiding) 
is isomorphic to the identity functor as a tensor functor.
\end{lemma}

\begin{proof} It was proved in \cite[4.3.2]{O} that the category ${\rm Ver}_{p}$ is equivalent to the
semisimple quotient of the category of tilting modules over $SL(2)$.
The result follows from the universal property of the category
of tilting modules over $SL(2)$, see \cite[Theorem 2.4]{O1} (this is equivalent to the description
of tilting modules over $SL(2)$ as the idempotent completion of the Temperley-Lieb category).
\end{proof}

\begin{corollary} \label{injiso}
Let $\A$ be a fusion category and let $G_1, G_2$ be two injective tensor
functors $\A \to {\rm Ver}_{p}$. Then $G_1$ and $G_2$ are isomorphic as tensor
functors.
\end{corollary}

\begin{proof} The classification of fusion subcategories of ${\rm Ver}_{p}$ 
(see \cite[Proposition 3.3]{O})
implies that the images $G_1(\A)\subset {\rm Ver}_{p}$ and $G_2(\A)$ are the same. 
Thus there exists a tensor
autoequivalence $R$ of $\A$ such that $G_2\simeq G_1\circ R$. Again by the classification 
of fusion subcategories of ${\rm Ver}_{p}$ there exists a fusion subcategory
$\B \subset {\rm Ver}_{p}$ such that ${\rm Ver}_{p}=G_1(\A)\boxtimes \B$. Thus the
autoequivalence $R$ extends to autoequivalence of ${\rm Ver}_{p}$ as $R\boxtimes \id$.
Then Lemma \ref{autover} implies that $R\simeq {\rm id}$ and we are done.
\end{proof}

\subsection{Proof of Theorem \ref{ufiber}}
Let $F_1$ and $F_2$ be two symmetric tensor functors $\C \to {\rm Ver}_{p}$. Let 
$I_1, I_2$ be right adjoint functors and let $A_1=I_1(\be), A_2=I_2(\be)$ be the
corresponding commutative algebras, see Section \ref{commalg}. Then $F_2(A_1)$ is a
unital, associative, commutative algebra in ${\rm Ver}_{p}$. Thus by Proposition \ref{homto1}
there exists a unital algebra homomorphism $F_2(A_1)\to \be$. Thus we have
a unital algebra homomorphism $\tilde I_2(F_2(A_1))\to \tilde I_2(\be)=A_2$, or 
$A_1\otimes A_2\to A_2$, see Proposition \ref{falg} (iii). Composing this with the homomorphism
$A_1=A_1\otimes \be \to A_1\otimes A_2$ we get a unital (and hence non-zero) 
homomorphism $A_1\to A_2$. By Corollary \ref{noid}, this homomorphism is injective. 
By symmetry we have an injective homomorphism $A_2\to A_1$. It follows that the algebras
$A_1$ and $A_2$ are isomorphic. 

We choose an isomorphism and set $A=A_1\simeq A_2$. By Corollary \ref{compo} the functors
$F_i, i=1,2$ are isomorphic to $\tilde F_i\circ F_A$ where $\tilde F_i$ are injective tensor functors
$\C_A\to {\rm Ver}_{p}$. By Corollary \ref{injiso} the functors $\tilde F_1$ and $\tilde F_2$ are
isomorphic. Hence $F_1$ and $F_2$ are isomorphic and the proof is complete.

\section{Super Frobenius-Perron Dimension}

\subsection{Definition of ${\rm SFPdim}$} Let $\C$ be a  
symmetric fusion category over $\kk$, and $F : \C \rightarrow {\rm Ver}_{p}$
be the Verlinde fiber functor. For any $X \in \C$ we can write
$$
F(X) = F(X)^{+} \oplus F(X)^{-}
$$
where $F(X)^{+} \in {\rm Ver}_{p}^{+}$ 
and $F(X)^{-} \in {\rm Ver}_{p}^{-}$. 

\begin{definition} Let
$$
d_{+}(X) := \FPdim(F(X)^{+}), \;\; d_{-}(X) := \FPdim(F(X)^{-}),
$$
Define the super Frobenius-Perron dimension of $X$ as 
$$
{\rm SFPdim}(X) := d_{+}(X) - d_{-}(X).
$$
\end{definition} 

\begin{remark} 
1. Note that this definition makes sense in characteristic zero, if we replace
${\rm Ver}_p$ with ${\rm sVec}$, and $F$ with the usual fiber functor; 
in this case, we recover the standard notion of superdimension of a supervector space.   
Also, if $\C = {\rm sVec}$ and $F$ is the natural embedding of 
${\rm sVec}$ into ${\rm Ver}_{p}$ as the subcategory generated by $L_{1}$ and $L_{p-1}$, 
then ${\rm SFPdim}$ also agrees with the usual notion of superdimension. 
In particular, ${\rm SFPdim}={\rm FPdim}$ in characteristic $2$ 
and coincides with the usual superdimension in characteristic $3$. 

2. Below we will give another definition of ${\rm SFPdim}$ which does not use $F$. 

3. It follows from the definition that ${\rm SFPdim}(X)={\rm SFPdim}(F(X))$. 
\end{remark}

\begin{lemma} \label{character}
(i) ${\rm SFPdim}$ is a character of the Grothendieck ring of $\C$. 

(ii) For $\C = {\rm Ver}_{p}$ and $F = {\rm id}$, we have
${\rm SFPdim}=\chi_{p-1}$. Thus, ${\rm SFPdim}(X)=\chi_{p-1}(F(X))$. 
\end{lemma}

\begin{proof} Using the equivalence
$$
{\rm Ver}_{p} \cong {\rm Ver}_{p}^{+} \boxtimes {\rm sVec},
$$
for $\C = {\rm Ver}_{p}$ and $F = {\rm id}$ we can write ${\rm SFPdim}$ as $\FPdim \boxtimes \sdim$, where $\sdim$ denotes the superdimension of a supervector space. 
Since $\FPdim$ and $\sdim$ are characters of the Grothendeick rings of ${\rm Ver}_{p}^{+}$ and ${\rm sVec}$ respectively, 
${\rm SFPdim}$ is a character of the Grothendeick ring of ${\rm Ver}_{p}$. Moreover, we have ${\rm SFPdim}(L_2)=-{\rm FPdim}(L_2)=-q-q^{-1}=q^{p-1}+q^{-p+1}=\chi_{p-1}(L_2)$,
which implies that ${\rm SFPdim}=\chi_{p-1}$ (as $L_2$ generates the Grothendieck ring of ${\rm Ver}_p$). Now the general case follows from the facts that 
${\rm SFPdim}(X)={\rm SPdim}(F(X))$ and that $F$ induces a homomorphism of Grothendieck rings.  
\end{proof}

\subsection{The second Adams operation and ${\rm SFPdim}$}

Let ${\rm char}\kk\ne 2$. For a symmetric tensor category $\C$ over $\kk$, the second Adams operation $\Psi^{2}$ is defined on the Grothendeick ring of $\C$ as 
$$
\Psi^{2}(X) = S^{2}X - \wedge^{2}X,
$$ 
the difference between the symmetric and exterior squares. 

\begin{lemma}\label{ringhomo} $\Psi^2$ is a ring homomorphism ${\rm Gr}(\C)\to {\rm Gr}(\C)$. 
\end{lemma} 

\begin{proof} We have $S^2(X\oplus Y)=S^2X\oplus X\otimes Y\oplus S^2Y$ and 
$\wedge^2(X\oplus Y)=\wedge^2X\oplus X\otimes Y\oplus \wedge^2Y$, which implies 
that $\Psi^2$ is a group homomorphism. 

We have $S^2(X\otimes Y)=S^2X\otimes S^2Y\oplus \wedge^2X\otimes \wedge^2Y$
and $\wedge^2(X\otimes Y)=S^2X\otimes \wedge^2Y\oplus \wedge^2X\otimes S^2Y$, 
which implies that $\Psi^2$ is multiplicative. 
\end{proof} 

Now we give another definition of ${\rm SFPdim}$ which does not use the functor $F$. 

\begin{proposition} Let $\C$ be a fusion category over $\kk$, and $X \in \C$. Then,
\begin{equation}\label{equa}
{\rm SFPdim}(X) = g(\FPdim(\Psi^{2}(X)))
\end{equation}
where $g : \Q(q) \rightarrow \Q(q)$ 
is the Galois automorphism defined by sending $q^{2}$ to $-q$ (note that $\Q(q) = \Q(q^{2})$).  
\end{proposition}

\begin{proof} Let $F$ be the Verlinde fiber functor $\C \rightarrow {\rm Ver}_{p}$. Then ${\rm SFPdim}(X) = {\rm SFPdim}(F(X))$. Also, since $F$ preserves the symmetric structure and the FP-dimension, $g(\FPdim(\Psi^{2}(X))) = g(\FPdim(\Psi^{2}(F(X)))).$ Hence, without loss of generality, we can assume that $\C = {\rm Ver}_{p}$ and $F={\rm id}$. 

By Lemma \ref{ringhomo} and Lemma \ref{character}, 
both sides of \eqref{equa} are characters on the complexified Grothendieck ring 
of ${\rm Ver}_{p}$. Hence, as $L_{2}$ tensor generates ${\rm Ver}_{p}$, it suffices to check the equality on $L_{2}$. In this case, if $p \ge 5$, the LHS is by definition $-q-q^{-1}$ and the RHS of \eqref{equa} is 
$$
g(\FPdim(L_{3}-L_{1})) = g([3]_{q} - 1) = g(q^{2} + q^{-2}) = -q-q^{-1},
$$
as desired. For $p = 3$, on both sides of \eqref{equa} we have $-1$ (here $g$ is the identity and $L_{3}=0$).
\end{proof}

\subsection{Symmetric and braided categorifications of fusion rings of rank 2} 

Here is an application of the second Adams operation. Let $R_n$ be the fusion ring of rank $2$ 
with basis $1,X$ and relation $X^2=nX+1$, where $n\ge 0$. 

\begin{proposition}\label{classifi} Let $n\ge 1$. There is a symmetric fusion category $\C$ with Grothendieck ring $R_n$ only 
if $n=1$ and ${\rm char}\kk=5$, and in this case  $\C={\rm Ver}_5^+$. 
\end{proposition} 

\begin{proof} If ${\rm char}\kk=2$ then there is a Verlinde fiber functor $F: \C\to {\rm Vec}$, so 
$\FPdim(X)$ must be an integer, i.e. the equation $x^2=nx+1$ should have integer solutions, which is 
impossible. So we may assume that ${\rm char}\kk\ne 2$. We have a homomorphism
$\Psi^2: R_n\to R_n$, so $\Psi^2(X)=X$ or $\Psi^2(X)=n-X$. But the multiplicity of $1$ 
in $\Psi^2(X)$ is $1$ or $-1$, as either $S^2X$ or $\wedge^2X$ contains 
a copy of the unit object $\bold 1$ (but not both). So $n=1$ and $\FPdim(X)=\frac{1+\sqrt{5}}{2}$. 
Also, $\C$ must be a degenerate category (otherwise, it would lift to a symmetric category in characteristic zero by \cite{EGNO}, Subsection 9.16, and hence would have had integer 
FP dimensions, see \cite{EGNO}, Theorem 9.9.26), so ${\rm char}\kk=5$ 
(as $\dim(\C)=\frac{5\pm \sqrt{5}}{2}$, which vanishes only in characteristic 
$5$). Thus we have a Verlinde fiber functor $\C\to {\rm Ver}_5$, which is easily seen to be an equivalence onto ${\rm Ver}_5^+$.  
\end{proof} 

Note that Proposition \ref{classifi} completely classifies symmetric categorifications of fusion rings of rank $2$, 
as the problem of categorification of the ring $R_0$ is trivial (since $X$ is invertible). 

The main result of \cite{O2} states that in characteristic zero there is no fusion categories with Grothendieck ring $R_n$, $n\ge 2$. It is an open question whether this is true in positive
characteristic. However we have the following

\begin{proposition}\label{nobra} There is no braided fusion categories $\C$ over $\kk$ with Grothendieck ring
$R_n$, $n\ge 2$. 
\end{proposition}

\begin{proof}
For the sake of contradiction assume that $\C$ is a braided fusion category with 
${\rm Gr}(\C)=R_n$, $n\ge 2$.
Then the same argument as in proof of \cite[Corollary 2.2]{O2} shows that $\C$ is spherical. 
The lifting theory (see \cite[9.16]{EGNO})
shows that $\C$ must be degenerate, that is $1+d^2=0$ where $d\in \kk$ is the dimension
of the object $X$ (as by the result of \cite{O2}, there is no categorifications of $R_n$ in characteristic zero).  Since $d^2=1+nd$, this is equivalent to $nd=-2$. 
Moreover by \cite[Proposition 2.9]{O} the ring ${\rm Gr}(\C)\otimes \kk$ must not be semisimple,
so $X\mapsto d$ is a unique homomorphism $R_n\to \kk$. Thus $\frac{s_{X,X}}{d}=d$ where
$s_{X,X}$ is the entry of $S-$matrix of $\C$, see \cite[Proposition 8.3.11]{EGNO}.
By \cite[Proposition 8.13.8]{EGNO} we have $s_{X,X}=\theta^{-2}(1+\theta nd)$ where
$\theta$ is the twist of the object $X$, see \cite[Definition 8.10.1]{EGNO}. Thus
$$\frac{s_{X,X}}{d}=d\Leftrightarrow \theta^{-2}(1+\theta nd)=d^2\Leftrightarrow 
\theta^{-2}(1-2\theta)=-1\Leftrightarrow \theta =1.$$
It follows (see \cite[Definition 8.10.1]{EGNO}) that the category $\C$ is symmetric, and the
result follows from Proposition \ref{classifi}.
\end{proof}

\begin{remark} The argument in the proof of Proposition \ref{nobra} works also in the case $n=1$, 
showing that if $\C$ is degenerate then it must be symmetric, and hence by Proposition \ref{classifi} 
has to be equivalent to ${\rm Ver}_5^+$ with ${\rm char}\kk=5$. Another option for $n=1$ 
is that $\C$ is nondegenerate; then by lifting theory (\cite[9.16]{EGNO}) it is a reduction modulo $p={\rm char}\kk$ of a Yang-Lee category $YL_+$ or $YL_-$ in characteristic zero (\cite[Exercise 9.4.6]{EGNO}), 
and each can be taken with two different braidings (inverse to each other). Thus we have four choices 
for each characteristic $p\ne 5$ (which are distinct since lifting is faithful). 

Also, the case $n=0$ is easy since $X$ is invertible. Namely, in this case, if $p=2$, then 
$\C={\rm Vec}_{\Bbb Z/2\Bbb Z}$ ($\Bbb Z/2\Bbb Z$-graded vector spaces), and the only possible braiding is trivial. On the other hand, for $p>2$ the classification is the same as in characteristic zero,
namely by quadratic forms on $\Bbb Z/2\Bbb Z$, which are labeled by solutions of the equation $\theta^4=1$ in $\kk$ (\cite[8.2]{EGNO}). Thus, we have obtained a full classification of braided fusion categories 
with two simple objects over a field of any characteristic.  
\end{remark} 

\section{Decomposition of the Verlinde fiber functor of an object of a symmetric fusion category}

\subsection{The decomposition formulas}
Let $\C$ be a symmetric fusion category over $\kk$ and let $F: \C \rightarrow {\rm Ver}_{p}$ be the Verlinde fiber functor. Let $X \in \C$. Our goal in this section is to compute the decomposition of $F(X)$ into the simples of ${\rm Ver}_{p}$ in terms of $\FPdim(X)$ and ${\rm SFPdim}(X)$. Let
$$
F(X) = \bigoplus_{r=1}^{p-1} a_{r}L_{r}.
$$
If $p=2$ then $a_1=\FPdim(X)$, so assume that $p>2$. 

Let ${\rm Tr}: K\to \Bbb Q$ be the trace map, given by ${\rm Tr}(x)=\sum_{s=0}^{\frac{p-3}{2}}g_s(x)$. 
For example, for $r$ not divisible by $p$, we have 
\begin{equation}\label{tracefor}
{\rm Tr}(q^r+q^{-r})=(-1)^{r-1}.
\end{equation} 

\begin{theorem} \label{decomposition}
For $p>2$ we have
$$
a_{r} = \frac{1}{2p} {\rm Tr}\left((q^{-r}-q^r)(q - q^{-1}) (\FPdim(X) - (-1)^{r}{\rm SFPdim}(X))\right).
$$
In other words, for $r$ odd, we have 
$$
a_{r} = \frac{1}{p}{\rm Tr}((q^{-r}-q^r)(q - q^{-1})d_{+}(X))
$$
and for $r$ even, we have
$$
a_{r} = \frac{1}{p} {\rm Tr}((q^{-r}-q^r)(q - q^{-1})d_{-}(X)).
$$
\end{theorem}

\begin{proof} Applying $g_s\circ \FPdim$ and $g_s\circ \SFPdim$ to the decomposition of $F(X)$, we get, for each $s = 0, \ldots, \frac{p-3}{2}$:
$$
g_{s}(\FPdim(X)) = \sum_{r=1}^{p-1} a_{r}[r]_{q^{2s + 1}}
$$
and
$$
g_{s}({\rm SFPdim}(X)) = \sum_{r=1}^{p-1} (-1)^{r-1} a_{r}[r]_{q^{2s + 1}}.
$$
Clearing denominators, we get 
$$
g_{s}((q-q^{-1})\FPdim(X)) = \sum_{r=1}^{p-1} a_{r}(q^{(2s + 1)r}-q^{-(2s+1)r})
$$
and
$$
g_{s}((q-q^{-1}){\rm SFPdim}(X)) = \sum_{r=1}^{p-1} (-1)^{r-1} a_{r}(q^{(2s + 1)r}-q^{-(2s+1)r}).
$$
Let us extend the sequence $a_r$ to all integer $r$ by setting $a_{r} = -a_{-r} = a_{r + 2p}$ 
(in particular, $a_r=0$ if $r$ is divisible by $p$). Then we can rewrite the above sums as summations over 
$\mathbb{Z}/2p\mathbb{Z}$:
$$
g_{s}((q - q^{-1})\FPdim(X)) = \pm \sum_{r \in {\Bbb Z}/2p{\Bbb Z}} a_{r}q^{\pm (2s+1)r}
$$
and
$$
-g_{s}((q - q^{-1}){\rm SFPdim}(X)) = \pm \sum_{r \in {\Bbb Z}/2p{\Bbb Z}} a_{r}q^{\pm (p + 2s + 1)r}.
$$
Now taking the inverse Fourier transform on the group ${\Bbb Z}/2p{\Bbb Z}$ gives the desired result. 
\end{proof}

\subsection{Transcendence degrees} 
Let $X$ be an object of a fusion category $\C$ over $\kk$. 
Recall from \cite{Ven} that the algebra of invariants $(SX)^{\rm inv}$ 
(which is an ordinary commutative $\kk$-algebra) is finitely generated. Let ${\rm Trd}_+(X)$ be the transcendence degree of this algebra, 
i.e. the largest number of algebraically independent elements. Similarly, let ${\rm Trd}_-(X)$ be the transcendence degree of $(\wedge X)_{\rm even}^{\rm inv}$ 
(the largest number of even algebraically independent elements). 
It is shown in \cite{Ven} that ${\rm Trd}_+(X)=d$ if and only if 
$SX$ is a finitely generated module over $\kk[z_1,...,z_d]\subset (SX)^{\rm inv}$, 
and likewise ${\rm Trd}_-(X)=d$ if and only if $\wedge X$ 
 is a finitely generated module over $\kk[z_1,...,z_d]\subset (\wedge X)_{\rm even}^{\rm inv}$. 
This implies that ${\rm Trd}_\pm(X)$ are preserved under symmetric tensor functors
between fusion categories.

It follows from Proposition \ref{prop31} that ${\rm Trd}_+(L_i)=0$ if $2\le i\le p-1$ and   
${\rm Trd}_-(L_i)=0$ if $1\le i\le p-2$, while ${\rm Trd}_+(L_1)={\rm Trd}_-(L_{p-1})=1$.  
Therefore, we have 

\begin{proposition}\label{tradeg} 
(i) One has ${\rm Trd}_+(X)=a_1$ and ${\rm Trd}_-(X)=a_{p-1}$. 

(ii) One has  
$$
{\rm Trd}_+(X) = \frac{1}{p}{\rm Tr}(|q - q^{-1}|^2d_{+}(X)),
$$
and 
$$
{\rm Trd}_-(X) = \frac{1}{p} {\rm Tr}(|q - q^{-1}|^2d_{-}(X)).
$$
\end{proposition} 

\begin{proof}
(i) is immediate from the above, and (ii) follows from (i) and Theorem \ref{decomposition} by plugging in $r=1,p-1$.   
\end{proof} 

\section{Decomposition of symmetric powers of simple objects in ${\rm Ver}_{p}$}

We will now apply Theorem \ref{decomposition}  to compute the decomposition 
into simples of $S^iL_{m}$, the symmetric powers of the simple object $L_{m} \in {\rm Ver}_{p}$, for each $ 2\le m\le p-1$ and $0\le i\le p-m$ 
(note that by Proposition \ref{prop31},  we have $S^iL_m=0$ if $i>p-m$). To do so, we need to compute the $\FPdim$ and ${\rm SFPdim}$ of $S^{i}L_{m}$.  

As before, we assume that $p>2$. Let 
$$
\binom{n}{m}_{z}:=\frac{\prod_{j=1}^m (z^{n-j+1}-z^{-n+j-1})}{\prod_{j=1}^m (z^j-z^{-j})}
$$
be the symmetrized Gauss polynomial (the $z$-binomial coefficient).  

\begin{proposition}\label{bino} (i) For $2\le m\le p-1$ and $0\le i \le p-m$, 
$$
\FPdim(S^{i}L_{m}) = \binom{i + m-1}{m-1}_{q},
$$

(ii) One has $\SFPdim(S^iL_m)=(-1)^{i(m-1)}\FPdim(S^iL_m)$. 
\end{proposition} 

\begin{proof} (i) Let ${\rm Ver}_p(SL_m)$ be the Verlinde category attached to the group $SL_m$ (see \cite{O}, 4.3.2, 4.3.3 and references therein). The $\Bbb Z_+$-basis of ${\rm Gr}({\rm Ver}_p(SL_m))$ is $\lbrace{V_\lambda\rbrace}$, where $\lambda$ runs through dominant integral weights for $SL_m$ such that $(\lambda+\rho,\theta)<p$, and we have a character of this ring given by 
\begin{equation}\label{qdim}
V_\lambda\mapsto \dim_q V_\lambda=\prod_{\alpha>0} \frac{[(\lambda+\rho,\alpha)]_q}{[(\rho,\alpha)]_q},
\end{equation} 
where $\alpha$ runs through positive roots, $\theta$ is the maximal root, and $\rho$ the half-sum of the positive roots (the $q$-deformed Weyl dimension formula, see \cite{BK}, 3.3). It is easy to see that all these values are positive.

Let $V=V_{\omega_1}$, where $\omega_1$ is the first fundamental weight, i.e., $V$ is the tautological object. Then $S^iV=V_{i\omega_1}$ for $i\le p-m$. 
So, by \eqref{qdim}, $\dim_q S^iV=\binom{i + m-1}{m-1}_{q}$. Hence 
$\FPdim(S^iV)= \binom{i + m-1}{m-1}_{q}$.  

We have the Verlinde fiber functor $F: {\rm Ver}_p(SL_{m})\to {\rm Ver}_p(SL_2)$, and $F(V)=L_m$. Hence $F(S^iV)=S^iL_m$, and the statement follows. 

(ii) It is clear that $S^iL_m$ is in ${\rm Ver}_p^+$ if $i(m-1)$ is even and is in ${\rm Ver}_p^-$ otherwise, which 
implies the statement. 
\end{proof} 

Let $S^iL_m=\oplus_{r=1}^{p-1} a_rL_r$.

\begin{corollary}  
If $i(m-1)-r$ is odd, then 
$$
a_{r} = \frac{1}{p}{\rm Tr}\biggl((q^{-r}-q^r)(q - q^{-1})\binom{i+m-1}{m-1}_q\biggr),
$$
and if $i(m-1)-r$ is even, then $a_r=0$.
In particular,  
\begin{equation}\label{bf}
\dim_\kk (S^iL_m)^{\rm inv}=\frac{1}{p}{\rm Tr}\biggl(|q - q^{-1}|^2\binom{i+m-1}{m-1}_q\biggr).
\end{equation}
\end{corollary} 

\begin{proof} This follows from Proposition \ref{bino} and Theorem \ref{decomposition}. 
\end{proof} 

These formulas can be written more combinatorially using that for any Laurent polynomial 
$f(z)$ with integer coefficients such that $f(z)=f(z^{-1})$, we have 
$$
2{\rm Tr} f(q)=\tau(f), 
$$
where 
$$
\tau(\sum_{j\in \Bbb Z} b_jz^j)=p\sum_{j\in \Bbb Z}(-1)^jb_{pj}-\sum_{j\in \Bbb Z}(-1)^jb_j=
p\sum_{j\in \Bbb Z}(-1)^jb_{pj}-f(-1), 
$$
which follows from \eqref{tracefor}. For instance, from \eqref{bf} we have 
\begin{equation}\label{bf1}
\dim_\kk (S^iL_m)^{\rm inv}=\sum_{j\in \Bbb Z}(-1)^jb_{pj},
\end{equation}
where 
$$
\sum_{j\in \Bbb Z} b_jz^j=-\frac{1}{2}(z-z^{-1})^2\binom{i+m-1}{m-1}_z.
$$
More generally, we have 
$$
a_r=\sum_{j\in \Bbb Z}(-1)^jb_{pj,r},
$$
where 
$$
\sum_{j\in \Bbb Z} b_{j,r}z^j=\frac{1}{2}(z^{-r}-z^r)(z-z^{-1})\binom{i+m-1}{m-1}_z.
$$
\begin{remark} Since ${\rm Ver}_p={\rm Ver}_p(SL_2)$, formula \eqref{bf1} may be viewed as a ``fusion'' analog 
of the classical Cayley-Sylvester formula for the number $N(i,m)$ of linearly independent invariants of degree $i$ of a binary form of degree $m$, 
which says that $N(i,m)=b_0$ (see \cite{S}). Indeed, this formula is recovered from \eqref{bf1} as $p\to \infty$.   
\end{remark}

\begin{remark} Let $G$ be a simply connected simple algebraic group over $\kk$ (for simplicity assumed of type ADE), $h$ the Coxeter number of $G$, and $p>h$. Let ${\rm Ver}_p(G)$ be the Verlinde category 
corresponding to $G$ (see \cite{O}, 4.3.2, 4.3.3). The simple objects of ${\rm Ver}_p(G)$ are highest weight modules $V_\lambda$, where $(\lambda+\rho,\theta)<p$. Using a similar method to the above, we can compute the decomposition of $F(V_\lambda)$. Namely, similarly to the $SL_m$ case, 
$\FPdim(V_\lambda)=\dim_q(V_\lambda)$, given by formula \eqref{qdim} (this is a character and all its values 
are positive). Also, for each $\lambda$, $F(V_\lambda)$ is in ${\rm Ver}_p^+$ or ${\rm Ver}_p^-$ 
depending on whether the element $-1=\exp(2\pi i h_\rho)\in SL_2^{\rm principal}\subset G$ 
acts by $1$ or $-1$ on $V_\lambda$ in characteristic zero (see \cite{O}, 4.3.3). In other words, 
we have $\SFPdim(V_\lambda)=(-1)^{(\lambda,\rho)}\FPdim(V_\lambda)$. 
Thus we get $F(V_\lambda)=\sum_r a_rL_r$, where 
$$
a_r=\frac{1}{p}\Tr\left((q^{-r}-q^r)(q-q^{-1})\prod_{\alpha>0} \frac{[(\lambda+\rho,\alpha)]_q}{[(\rho,\alpha)]_q}\right)
$$
\end{remark} 

\section{$p$-adic dimensions in a fusion category} 

The following proposition relates the $p$-adic dimensions with the super Frobenius Perron dimension. 

\begin{proposition} One has 
$$
\SFPdim(X)=(\Dim_+\boxtimes \FPdim)({\rm Fr}^n(X))=(\Dim_-\boxtimes \FPdim)({\rm Fr}^n(X)).
$$
where $n$ is such that ${\rm Fr}^n: \C\to {\mathcal D}\boxtimes {\rm Ver}_p^+$, with ${\mathcal D}\subset \C^{(n)}$ nondegenerate (see Subsection \ref{vff}). 
\end{proposition} 

\begin{proof} Since ${\mathcal{D}}$ is a nondegenerate fusion category, it  
admits a super fiber functor, and thus both $p$-adic dimensions in ${\mathcal{D}}$ coincide with the usual superdimension.
This implies the statement. 
\end{proof} 

\begin{theorem}\label{decomp} (i) $\Dim_+(L_r)$ equals $1$ if $r=1$ and $r-p$ if $r>1$. 
$\Dim_-(L_r)$ equals $r$ if $r<p-1$, and $-1$ if $r=p-1$. 

(ii) The $p$-adic dimensions of any object $X$ of a fusion category are as follows: if 
$F(X)=\sum_{r=1}^{p-1} a_rL_r$ then
$$
\Dim_+(X)=a_1+\sum_{i>1}(i-p)a_i,\ \Dim_-(X)=\sum_{i<p-1} ia_i-a_{p-1}.
$$
\end{theorem} 

\begin{proof} (i) Let $r>1$. By Proposition \ref{prop31}, $(1-z)^{-\Dim_+(L_r)}$ is a polynomial 
of degree $\le p-r$. Thus, $-\Dim_+(L_r)$ is an integer between $0$ and $p-r$. Since 
$\Dim_+(L_r)$ must equal $\dim L_r=r$ modulo $p$, we conclude that $\Dim_+(L_r)=r-p$. 

Let $r<p-1$. By Proposition \ref{prop31}, $(1+z)^{\Dim_-(L_r)}$ is a polynomial 
of degree $\le r$. Thus, $\Dim_-(L_r)$ is an integer between $0$ and $r$. Since 
$\Dim_-(L_r)$ must equal $\dim L_r=r$ modulo $p$, we conclude that $\Dim_-(L_r)=r$. 

The remaining cases $r=1,p-1$ are easy, as the corresponding objects are invertible. 

(ii) is immediate from (i).    
\end{proof} 

\begin{corollary} 
One has 
$$
{\rm length}F(X)={\rm Trd}_+(X)+{\rm Trd}_-(X)+\frac{\Dim_-(X)-\Dim_+(X)}{p}.
$$
\end{corollary} 

\begin{proof} This follows from Theorem \ref{decomp}(ii), using that 
${\rm length}F(X)=\sum_i a_i$.  
\end{proof} 

\section{Classification of triangular semisimple Hopf algebras in positive characteristic} 

The following theorem classifies traingular semisimple Hopf algebras in any characteristic, generalizing  
Theorem 2.2 of \cite{EG} in characteristic zero. 

For a finite group scheme $G$ over $\kk$, let $\kk G$ denote the dual Hopf algebra 
$O(G)^*$ to the function algebra $O(G)$ of $G$. 

\begin{theorem} \label{classif} Let $H$ be a finite dimensional triangular semisimple Hopf algebra over a field $\kk$ of any characteristic (see \cite{EGNO}, Chapter 5). Then there exists a semisimple finite group scheme $G$ over $\kk$, a central element $\varepsilon\in G(\kk)$ of order $\le 2$ (with $\varepsilon=1$ if $p=2$) and a twist $J\in \kk G\otimes \kk G$
such that $H=\kk G^J$ as a Hopf algebra, and the universal $R$-matrix of $H$ 
is $R=J_{21}^{-1}R_0J$, where $R_0=1\otimes 1$ if $p=2$ and 
$R_0=\frac{1}{2}(1\otimes 1+1\otimes \varepsilon+\varepsilon\otimes 1-\varepsilon\otimes \varepsilon)$ otherwise. 
\end{theorem}

Note that this classification is completely explicit.
Namely, by Nagata's theorem (\cite{DG}, IV, 3.6), $G=\Gamma\ltimes A$, where 
$\Gamma$ is a finite group of order coprime to $p$ and $A$ is the dual group scheme to a finite abelian $p$-group $P$; in other words, we have $\kk G=\kk \Gamma\ltimes O(P)$. 
Moreover, the twists $J$ for $G$ are classified in terms 
of subgroup schemes of $G$ and 2-cocycles on them 
in \cite{Ge}, Proposition 6.3. Namely, a twist
corresponds to a nondegenerate class $\psi\in H^2(G',\Bbb G_m)$, where 
$G'$ is a group subscheme of $G$. It is easy to see that $G'$ is conjugate to $\Gamma'\ltimes A'$, where 
$\Gamma'\subset \Gamma$ and $A'\subset A$ is $\Gamma'$-invariant. Now the Hochschild-Serre spectral sequence
(for group schemes) implies that $\psi$ is pulled back from $\Gamma'$ (as the cohomology $H^i(A',\Bbb G_m)$ vanishes for $i>0$, so the pullback map 
$H^i(\Gamma,\kk^\times)=H^i(\Gamma,\Bbb G_m)\to H^i(G,\Bbb G_m)$ is an isomorphism). 
Since $\psi$ is nondegenerate, this implies that $A'=1$, i.e., the twist $J$ comes from $\kk \Gamma\otimes \kk \Gamma$, 
i.e. corresponds to a subgroup of central type $\Gamma'\subset \Gamma$.   

\begin{proof} We may assume that ${\rm char}\kk=p>0$. 
For $p=2$, this is immediate from \cite{O} (see 
\cite{O}, Corollary 1.6), so let us assume that $p>2$. 
The category $\Rep H$ is a symmetric fusion category. 
Let $F:\Rep H\to \Ver_p$ be the Verlinde fiber functor. 
Since $\Rep H$ is integral (i.e., has integer FP dimensions), 
$F$ lands in $\sVec\subset {\rm Ver}_p$, as this is the largest 
integral subcategory of ${\rm Ver}_p$.  
Thus, as in \cite{O}, 1.3, $\Rep H=\Rep (G,\varepsilon)$, 
the category of representations of $G$ on superspaces on which $\varepsilon$ acts by parity, where $G$ is a 
semisimple finite supergroup scheme over $\kk$ and $\varepsilon\in G(\kk)$ 
is an element of order $\le 2$ acting on $O(G)$ by the parity automorphism.   
(For a general theory of affine algebraic supergroup schemes 
we refer the reader to \cite{Ma}.) 

It remains to show that $G$ is in fact an ordinary group scheme. 
This is shown in \cite{Ma} (see Corollary 43 and Theorem 45); here  
we provide an alternative proof. We will need the following two lemmas. 

\begin{lemma}\label{l1} If $M$ is a subsupergroup scheme of $G$ containing $\varepsilon$, then $\Rep(M,\varepsilon)$ is semisimple.
\end{lemma}

\begin{proof} Recall that if $\C$ is a fusion category and 
$\Phi: \C\to {\mathcal{D}}$ is a surjective tensor functor (\cite{EGNO}, 6.3)
to a finite tensor category, then ${\mathcal{D}}$ is also fusion. Indeed, ${\mathcal{D}}$ is an exact $\C$-module (\cite{EGNO}, Example 7.5.6), hence
semisimple. Now take $\C=\Rep(G,\varepsilon)$ and ${\mathcal{D}}=\Rep(M,\varepsilon)$, and take $\Phi$ to be the restriction functor.   
\end{proof} 

Now let $G_+$ be the even part of $G$ (i.e., $O(G_+)=O(G)/I$, where $I$ is the ideal generated by the odd elements, which is automatically a Hopf ideal),
and $G_+^0$ its connected component of the identity. By Lemma \ref{l1}, $\Rep(G_+^0)$ is semisimple. Thus, by Nagata's theorem (\cite{DG}, IV, 3.6), $G_+^0$ is abelian.

Let ${\rm Lie}(G)=\g_+\oplus \g_-$ be the Lie algebra of $G$, decomposed into the even and odd parts. Let $G^0$ be the connected component of the identity in $G$.

\begin{lemma}\label{l2} The action of $G_+^0$ on $\g_-$ is trivial. In other words, $G_+^0$ is central in $G^0$.
\end{lemma} 

\begin{proof} Let $x\in \g_-$ be an eigenvector, on which $G_+^0$ acts by a character $\chi$. 
We claim that $[x,x]\ne 0$. Indeed, assume the contrary, and let $N$ be the supergroup scheme generated by $\varepsilon$ and $\exp(\kk x)$. 
Then we have a surjective restriction functor $\Rep(G,\varepsilon)\to \Rep(N,\varepsilon)$, so $\Rep(N,\varepsilon)$ is semisimple by Lemma \ref{l1}, which is a contradiction. 

Now, $G_+^0$ acts on $[x,x]$ by $\chi^2$. But $[x,x]\in \g_+$, 
hence $\chi^2=1$, as $G_+^0$ is abelian. 

But $\chi$ has order which is a power of $p$, so $\chi=1$. Thus, Lemma \ref{l2} follows from Lemma \ref{l1} (which implies that the action of $G_+^0$ on $\g_-$ is semisimple).
\end{proof} 

We see that we have a central extension of supergroup schemes
$$
1\to G_+^0\to G^0\to \exp(\g_-)\to 1,
$$
where $\exp(\g_-)={\rm Spec}\wedge\g_-$ (purely odd abelian supergroup scheme). 
Let $P$ be the supergroup scheme generated by $\varepsilon$ and $G^0$, and $Q$ be the supergroup scheme generated by $\varepsilon$ and $\exp(\g_-)$. 
Thus the category $\Rep(Q,\varepsilon)$ is a tensor subcategory of
$\Rep(P,\varepsilon)$. But the category $\Rep(P,\varepsilon)$ is semisimple by Lemma \ref{l1}.
Hence, $\Rep(Q,\varepsilon)$ is semisimple. This can only hold if $\g_-$ is zero.
Thus, we have $G^0=G_+^0$, hence $G=G_+$, as desired. 
\end{proof}

\end{document}